\documentclass[12pt, leqno]{article}

\usepackage{amsmath,amsthm}
\usepackage{amssymb}

\usepackage{enumerate}

\newtheorem{thm}{Theorem}[section]
\newtheorem{cor}[thm]{Corollary}
\newtheorem{lem}[thm]{Lemma}
\newtheorem{prop}[thm]{Proposition}

\newtheorem*{nonumthm}{Theorem} 

\theoremstyle{definition}

\newtheorem{example}[thm]{Example}
\newtheorem*{ques}{Question} 

\numberwithin{equation}{section}

\renewcommand{\phi}{\varphi}
\newcommand{\eps}{\epsilon}

\renewcommand{\bar}{\overline}

\newcommand{\N}{\mathbb{N}}

\newcommand{\R}{\mathbb{R}}

\newcommand{\LF}{\mathcal{F}}
\newcommand{\LP}{\mathcal{P}}

\newcommand{\LB}{\mathcal{B}}

\newcommand{\LS}{\mathcal{S}}
\newcommand{\LI}{\mathcal{I}}

\DeclareMathOperator{\ran}{ran}

\newcommand{\LH}{\mathcal{H}}
\newcommand{\LK}{\mathcal{K}}
\newcommand{\rank}{\mathrm{rank}}
\newcommand{\ess}{\mathrm{ess}}

\DeclareMathOperator{\linspan}{span}

\begin{document}




\title{Borel equivalence relations in the space of bounded operators}

\author{Iian B. Smythe\\
Department of Mathematics\\ 
Cornell University\\
Ithaca, NY, USA, 14853\\
E-mail: ismythe@math.cornell.edu}

\date{September 17, 2016}

\maketitle


\renewcommand{\thefootnote}{}

\footnote{2010 \emph{Mathematics Subject Classification}: Primary 03E15, 47B10, Secondary 46A45.}

\footnote{\emph{Key words and phrases}: Borel equivalence relations, turbulence, compact operators, Calkin algebra, Schatten $p$-class, projection operators.}

\renewcommand{\thefootnote}{\arabic{footnote}}
\setcounter{footnote}{0}


\begin{abstract}
	We consider various notions of equivalence in the space of bounded operators on a Hilbert space, in particular modulo finite rank, modulo Schatten $p$-class, and modulo compact. Using Hjorth's theory of turbulence, the latter two are shown to be not classifiable by countable structures, while the first is not reducible to the orbit equivalence relation of any Polish group action. The results for modulo finite rank and modulo compact operators are also shown for the restrictions of these equivalence relations to the space of projection operators.
\end{abstract}

\section{Introduction}

A fundamental problem in the theory of operators on an infinite dimensional separable complex Hilbert space is to classify a collection of operators up to some notion of equivalence, a classical example being the following:

\begin{nonumthm}[Weyl--von Neumann \cite{vN35}]
	For $T$ and $S$ bounded self-adjoint operators on a Hilbert space as above, the following are equivalent:
	\begin{enumerate}[\upshape (i)]
		\item $T$ and $S$ are \emph{unitarily equivalent modulo compact}, i.e., there is a compact operator $K$ and a unitary operator $U$ such that $UTU^*-S=K$.
		\item $T$ and $S$ have the same \emph{essential spectrum}.
	\end{enumerate}
\end{nonumthm}

That is, bounded self-adjoint operators are completely classified up to unitary equivalence modulo compact by their essential spectra.

The modern theory of Borel equivalence relations affords us a general framework for such results. Given a space $X$ of objects, and an equivalence relation $E$ on $X$, completely classifying the elements of $X$ up to $E$-equivalence amounts to finding another space $Y$ with equivalence relation $F$, and specifying a map $f:X\to Y$ such that
\[
	xEy \quad\Leftrightarrow\quad f(x)Ff(y),
\]
for all $x,y\in X$. The spaces and equivalence relations should be ``reasonably definable'', in the sense that the former are Polish (or standard Borel) and the latter Borel. Enforcing that the classifying map $f$ is Borel captures that idea that $f$ is ``computing'' an invariant for the objects in $X$. Such a map is called a Borel reduction of $E$ to $F$, and its existence or non-existence allows us to compare the complexity of such equivalence relations. The ``simplest'' Borel equivalence relations are those given by equality on Polish spaces, and are said to be \emph{smooth}.

Recasting the motivating problem in this setting requires specifying a Polish or Borel structure on the collection of operators of interest, verifying that the notion of equivalence is Borel, and reducing the equivalence relation to another, preferably well-understood, equivalence relation. In the setting of the Weyl--von Neumann theorem above, we have:

\begin{nonumthm}[Ando--Matsuzawa \cite{MR3427601}]
	The map $T\mapsto\sigma_\ess(T)$ is a Borel function from the space of bounded self-adjoint operators to the Effros Borel space of closed subsets of $\R$. In particular, unitary equivalence modulo compact of bounded self-adjoint operators is smooth.
\end{nonumthm}

In contrast, many natural equivalence relations on classes of operators are not smooth. In fact, they exhibit a very strong form of non-classifiability; they cannot be reduced to the isomorphism relation on any class of countable algebraic or relational structures, e.g., groups, rings, graphs, etc. Such equivalence relations are said to be \emph{not classifiable by countable structures}. The method used to exhibit this property is Hjorth's theory of turbulence \cite{MR1725642}. Relevant examples are given by:

\begin{nonumthm}[Kechris--Sofronidis \cite{MR1855842}]
	Unitary equivalence of self-adjoint (or unitary) operators is not classifiable by countable structures.
\end{nonumthm}

\begin{nonumthm}[Ando--Matsuzawa \cite{MR3427601}]
	Unitary equivalence modulo compact of unbounded self-adjoint operators is not classifiable by countable structures.
\end{nonumthm}

In this article, we present non-classification results for collections of operators, focusing on equivalence relations induced by ideals of compact operators. The paper is arranged as follows: 

In \S 2, we review the relevant theory of bounded operators and Borel equivalence relations. In \S 3, we describe Borel and Polish structures on collections of operators, and in \S 4, establish the first of our results:

\begin{thm}\label{thm1}
	\begin{enumerate}
		\item Equivalence modulo finite rank operators (on $\LB(H)$ and $\LK(H)$) is a Borel equivalence relation that is not Borel reducible to the orbit equivalence relation of any Polish group action.
		\item Equivalence modulo compact operators (on $\LB(H)$ and $\LB(H)_{\leq 1})$) is a Borel equivalence relation that is not classifiable by countable structures.
		\item Equivalence modulo Schatten $p$-class (on $\LB(H)$, $\LB(H)_{\leq 1}$, and $\LK(H)$) is a Borel equivalence relation that is not classifiable by countable structures.
	\end{enumerate}
\end{thm}

As a consequence of Theorem \ref{thm1}(a), we show:

\begin{cor}\label{cor1}
	The space of all finite rank operators on $H$ is not Polishable in either the norm or strong operator topologies.
\end{cor}

In \S 5, we restrict our attention to the projection operators $\LP(H)$, considering the restrictions of modulo finite rank and modulo compact. The latter provides an alternate view of the projections in the Calkin algebra. We improve upon Theorem \ref{thm1} parts (a) and (c), showing:

\begin{thm}\label{thm2}
	\begin{enumerate}
		\item Equivalence modulo finite rank on $\LP(H)$ is not Borel reducible to the orbit equivalence relation of any Polish group action.
		\item Equivalence modulo compact on $\LP(H)$ is not classifiable by countable structures.
	\end{enumerate}
\end{thm}

\section{Preliminaries}

\subsection{Bounded operators on Hilbert spaces} 

Throughout, we fix an infinite dimensional separable complex Hilbert space $H$, with inner product $\langle\cdot,\cdot\rangle$. Let $\LB(H)$ denotes the set of all bounded operators on $H$, with operator norm $\|\cdot\|$. A standard reference for the theory of $\LB(H)$ is \cite{MR971256}. 

The \emph{strong operator topology} is the topology induced by the family of seminorms $T\mapsto \|Tv\|$ for $v\in H$, while the \emph{weak operator topology} is induced by the family of seminorms $T\mapsto |\langle Tv, w\rangle|$ for $v,w\in H$. 

	We denote by $T^*$ the (Hermitian) adjoint of an operator $T\in \LB(H)$. An operator $T\in\LB(H)$ is \emph{self-adjoint} if $T=T^*$, and \emph{positive} if $\langle Tv,v\rangle \geq 0$ for all $v\in H$. To each operator $T\in\LB(H)$, there is a unique positive operator $|T|$ satisfying $|T|^2=T^*T$.
	
	An operator $P\in\LB(H)$ is a \emph{projection} if $P^2=P^*=P$. Equivalently, $P$ is the orthogonal projection onto a closed subspace (namely, $\ran(P)$) of $H$. Every projection is positive with $\|P\|=1$ whenever $P\neq 0$. We denote the set of projections by $\LP(H)$.
		
	An operator $T\in\LB(H)$ is \emph{compact} if the image of the closed unit ball of $H$ under $T$ has compact closure. The set of compact operators is denoted by $\LK(H)$. $T$ is \emph{finite rank} if $\rank(T)=\dim(\ran(T))<\infty$, and the set of finite rank operators is denoted by $\LB_f(H)$. It is well-known that an operator on $H$ is compact if and only if it is a norm limit of finite rank operators. The following characterizes which diagonal operators are compact:

\begin{prop}[3.3.5 in \cite{MR971256}]\label{diag_cpct}
	If $T\in\LB(H)$ is diagonal with respect to an orthonormal basis $\{e_n:n\in\N\}$, say $Tv=\sum_{n=0}^\infty \lambda_n \langle v,e_n\rangle e_n$ for all $v\in H$, then $T$ is compact if and only if $\lim_{n\to\infty}\lambda_n=0$.
\end{prop}
	
	It is easy to check that $\LK(H)$ is a norm-closed, self-adjoint ideal in $\LB(H)$, and the corresponding quotient $\LB(H)/\LK(H)$ is called the \emph{Calkin algebra}.

For $1\leq p<\infty$, the \emph{Schatten $p$-class} $\LB^p(H)$ is the set of all operators $T\in\LB(H)$ such that for some orthonormal basis $(e_n)$ of $H$, one has $\sum_{n=0}^\infty\langle |T|^p e_n,e_n\rangle <\infty$ (this quantity is independent of the choice of basis). Each $\LB^p(H)$ is a self-adjoint ideal in $\LB(H)$, which fails to be norm-closed, and $\LB_f(H)\subsetneq \LB^p(H)\subsetneq \LK(H)$.

The following facts will be relevant in the sequel. We caution that the adjoint operation $T\mapsto T^*$ is not strongly continuous, and multiplication is not jointly strongly continuous on all of $\LB(H)$. For these facts, and the following lemma, see \S4.6 in \cite{MR971256}.

\begin{lem}\label{ops_cts}
	\begin{enumerate}
		\item The adjoint operation is weakly continuous on $\LB(H)$.
		\item Multiplication of operators is strongly continuous when restricted to $B\times \LB(H)\to\LB(H)$, where $B$ is any norm bounded subset of $\LB(H)$.
	\end{enumerate}
\end{lem}

\begin{lem}\label{strongly_closed_sets}
	\begin{enumerate}
		\item The closed unit ball $\LB(H)_{\leq 1}$ is strongly closed and completely metrizable in $\LB(H)$.
		\item The set of all self-adjoint operators $\LB(H)_{sa}$ is strongly closed in $\LB(H)$.
		\item The set of positive operators is strongly closed in $\LB(H)$.
		\item The set of projections $\LP(H)$ is strongly closed in $\LB(H)_{\leq 1}$.
	\end{enumerate}
\end{lem}

\begin{proof}
	For (a), $\LB(H)_{\leq 1}$ is closed by an application of the uniform boundedness principle, and completely metrizable by 4.6.2 in \cite{MR971256}. By Lemma \ref{ops_cts}(a), $\LB(H)_{sa}$ is weakly, thus strongly, closed, showing (b). Part (c) follows from the fact that the maps $T\mapsto\langle Tv,v\rangle$, for $v\in H$, are strongly continuous. For (d), let $\LI=\{T\in\LB(H)_{\leq1}:T^2=T\}$, and $\LS=\{T\in\LB(H)_{\leq1}:T^*=T\}$. By Lemma \ref{ops_cts}, $\LI$ and $\LS$ are strongly closed in $\LB(H)_{\leq 1}$, and $\LP(H)=\LS\cap \LI$.
\end{proof}

\subsection{Borel equivalence relations}

A \emph{Polish space} is a separable and completely metrizable topological space, while a \emph{standard Borel space} is a set $X$ together with a $\sigma$-algebra of Borel sets coming from some Polish topology on $X$. An equivalence relation $E$ on $X$ is \emph{Borel} if $\{(x,y)\in X^2:xEy\}$ is a Borel subset of $X^2$. Given equivalence relations $E$ and $F$ on Polish (or standard Borel) spaces $X$ and $Y$, respectively, a map $f:X\to Y$ is a \emph{Borel reduction} of $E$ to $F$ if $f$ is Borel measurable, and 
\[
	xEy \quad\Leftrightarrow\quad f(x)Ff(y)
\]
for all $x,y\in X$. Equivalently, $f$ is a Borel map which descends to a well-defined injection $X/E\to Y/F$. In this case, we say that $E$ is \emph{Borel reducible} to $F$, and write $E\leq_B F$. If $f$ is injective, we say that $f$ is a \emph{Borel embedding} of $E$ into $F$, and write $E\sqsubseteq_B F$. If $E\leq_B F$ and $F\leq_B E$, we write $E\equiv_B F$, and say that $E$ and $F$ are \emph{Borel bireducible}. Intuitively, $E\leq_B F$ means that classifying elements of $Y$ up to $F$ is at least as complicated as classifying elements of $X$ up to $E$, as any classification of the former yields one for the latter. $E\equiv_B F$ means that they are of equal complexity.

\begin{example}
	If $X$ is a Polish space, we denote by $\Delta(X)$ the equality relation on $X$. $\Delta(X)$ is a closed, and thus Borel, subset of $X^2$.
\end{example}

\begin{example}
	Identifying $2=\{0,1\}$, with the discrete topology, the Borel equivalence relation $E_0$ is defined on $2^\N$ by
	\[
		(x_n)_nE_0(y_n)_n \quad\Leftrightarrow\quad \exists m\forall n\geq m(x_n=y_n).
	\]
\end{example}

\begin{example}
	The Borel equivalence relation $E_1$ is defined on $\R^\N$ by
	\[
		(x_n)_n E_1 (y_n)_n \quad\Leftrightarrow\quad \exists m\forall n\geq m(x_n=y_n).	
	\]
\end{example}

A Borel equivalence relation $E$ on a Polish space $X$ is \emph{smooth} if $E\leq_B\Delta(Y)$ for some Polish space $Y$. Smooth equivalence relations are exactly those which admit complete classification by real numbers. It is well-known that $E_0$ is not smooth (cf.~\S6.1 in \cite{MR2455198}).

A \emph{Polish group} $G$ is a topological group which has a Polish topology. If $X$ is a Polish space, and $G$ acts continuously on $X$, i.e., the map $G\times X\to X$ given by $(g,x)\mapsto g\cdot x$ is continuous, then we say that $X$ is a \emph{Polish $G$-space}, and denote by $E_G$ (or sometimes $E/G$) the orbit equivalence relation
\[
	x\,E_G\,y \quad\Leftrightarrow\quad \exists g\in G(g\cdot x=y).
\]

A group with a given Borel structure (e.g.,~a Borel subgroup of a Polish group) is \emph{Polishable} if it can be endowed with a Polish group topology having the same Borel structure. It is easy to check that the orbit equivalence relation induced by the translation action of a Polishable (or Borel) subgroup of a Polish group is Borel. The following shows that $E_1$ is an obstruction to classification by orbits of Polish group actions.

\begin{thm}[Kechris--Louveau \cite{MR1396895}]\label{E1}
	Let $G$ be a Polish group, and $X$ a Polish $G$-space. Then, $E_1\not\leq_B E^X_G$.
\end{thm}

The isomorphism relation on the class of countable structures of a first-order theory, e.g., groups, rings, graphs, etc, can be represented as the orbit equivalence relation of a Polish $G$-space (cf.~Ch.~11 of \cite{MR2455198}). If an equivalence relation is Borel reducible to such a relation, we say that it is \emph{classifiable by countable structures}. Hjorth \cite{MR1725642} isolated a dynamical property of Polish $G$-spaces, called \emph{turbulence}, which implies that the corresponding orbit equivalence relation resists such classification.
	
\begin{thm}[Hjorth \cite{MR1725642}]\label{turbthm}
	Let $X$ be a Polish $G$-space. If the action of $G$ is turbulent, then $E_G$ is not classifiable by countable structures.
\end{thm}

For our purposes, it suffices to consider examples of these actions, and we omit a detailed discussion of turbulence (see \cite{MR2455198}, \cite{MR1725642}, \cite{MR1967835}).

\begin{lem}[Proposition 3.25 in \cite{MR1725642}, and p. 35 in \cite{MR1967835}]\label{transturb}
	\begin{enumerate}
		\item The translation action of $c_0$ on $\R^\N$ is turbulent.
		\item For $1\leq p<\infty$, the translation actions of $\ell^p$ on $\R^\N$ and $c_0$ are turbulent.
	\end{enumerate}
	In particular, the orbit equivalence relations $\R^\N/c_0$, $\R^\N/\ell^p$, and $c_0/\ell^p$ are not classifiable by countable structures.
\end{lem}

We consider the restrictions of $\R^\N/c_0$ and $\R^\N/\ell^p$ to the subset $[0,1]^\N$, and denote them by $[0,1]^\N/c_0$, and $[0,1]^\N/\ell^p$, respectively. It is evident that these are Borel equivalence relations, and that $[0,1]^\N/c_0\sqsubseteq_B\R^\N/c_0$, and $[0,1]^\N/\ell^p\sqsubseteq_B\R^\N/\ell^p$, via the inclusion maps. Moreover:

\begin{lem}[Lemma 6.2.2 in \cite{MR2441635}, see also \cite{MR2213716}]\label{transturb_restr}
	\begin{enumerate}
		\item $\R^\N/c_0 \equiv_B [0,1]^\N/c_0$.
		\item For $1\leq p<\infty$, $\R^\N/\ell^p\equiv_B[0,1]^\N/\ell^p$.
	\end{enumerate}
	In particular, $[0,1]^\N/c_0$ and $[0,1]^\N/\ell^p$ are not classifiable by countable structures.\footnote{Recent work by Hartz and Lupini \cite{Hartz_Lupini_2015} has developed a general theory of \emph{turbulent Polish groupoids} in which this can be seen more directly.}
\end{lem}

\section{Topology and Borel structure on $\LB(H)$}

In order to study Borel equivalence relations on $\LB(H)$ or its subsets, we must endow them with a Polish or standard Borel structure. The norm topology on $\LB(H)$ (or $\LP(H)$) is \emph{not} Polish as it contains discrete subsets of size $2^{\aleph_0}$: given an orthonormal basis $\{e_n:n\in\N\}$, consider the family of projections $P_x$ onto $\bar{\linspan}\{e_n:n\in x\}$, for $x\subseteq\N$. Instead, we focus on the strong operator topology.

\begin{lem}\label{Polish_std_Borel}
	\begin{enumerate}
		\item $\LB(H)_{\leq 1}$ and $\LP(H)$ are Polish in the strong operator topology.
		\item $\LB(H)$ is a standard Borel space with respect to the Borel structure generated by the strong operator topology.
	\end{enumerate}
\end{lem}

\begin{proof}
	(a) follows from Lemma \ref{strongly_closed_sets} parts (a) and (d), and the fact that the strong operator topology is separable. For (b), note that a countable union of standard Borel spaces is standard Borel, and $\LB(H) = \bigcup_{n\geq 1} n\LB(H)_{\leq 1}$.
\end{proof}

All references to Borel subsets of (or functions on) $\LB(H)$ will be with respect to this Borel structure, which coincides with that of the weak operator topology (as closed convex sets in $\LB(H)$ are weakly closed if and only if they are strongly closed, Corollary 4.6.5 in \cite{MR971256}). We caution that $\LB(H)$ is not Polish in the strong operator topology (it is not metrizable, see E 4.6.4 in \cite{MR971256}), nor is it even Polishable as a group with this Borel structure (this follows from Lemma 9.3.3 in \cite{MR2455198}).

The equivalence relations we study below arise from the ideals $\LB_f(H)$, $\LK(H)$ and $\LB^p(H)$ for $1\leq p<\infty$, and thus we will need to show that the corresponding ideal is Borel in the relevant topology.

\begin{lem}\label{rankleqn}
	For each $n\in\N$, the set $\LF_{\leq n}=\{T\in\LB(H): \rank(T)\leq n\}$ is strongly closed in $\LB(H)$.
\end{lem}

\begin{proof}\footnote{We thank the anonymous referee for a much shortened proof of this fact.}
	Suppose that $T\in\LB(H)$ is such that $\rank(T)>n$. There are vectors $v_0,\ldots,v_n\in H$ such that $Tv_0,\ldots,Tv_n$ are linearly independent, or equivalently, their Gram determinant $\det(\langle Tv_i,Tv_j\rangle_{i,j})$ is nonzero. Since the Gram determinant is continuous, there is a strongly open neighborhood of $T$ in $\LB(H)$ such that for all $S$ in that neighborhood, the Gram determinant $\det(\langle Sv_i,Sv_j\rangle_{i,j})$ is also nonzero, and so $\rank(S)>n$. Thus, the complement of $\LF_{\leq n}$ is strongly open.
\end{proof}

\begin{prop}\label{B_f(H)&K(H)Borel}
	$\LB_f(H)$ is an $F_\sigma$ set, and $\LK(H)$ is an $F_{\sigma\delta}$ set, in the strong operator topology on $\LB(H)$.
\end{prop}

\begin{proof}
	The claim for $\LB_f(H)$ is an immediate consequence of Lemma \ref{rankleqn}. The proof for the claim regarding $\LK(H)$ is essentially that of the more general Theorem 3.1~in \cite{MR0487448}. Let $\{T_k\}_{k=1}^\infty$ be a norm-dense sequence in $\LK(\LH)$, and let $B=\LB(\LH)_{\leq 1}$. Then,
	\[
		\LK(\LH) = \bigcap_{n=1}^\infty \left(\LK(\LH)+\frac{1}{n}B\right)\supseteq\bigcap_{n=1}^\infty\bigcup_{k=1}^\infty\left(T_k+\frac{1}{n}B\right) \supseteq \LK(\LH),
	\]
	where the first equality is the result of $\LK(H)$ being norm-closed in $\LB(H)$. Since $B$ is strongly closed, this shows that $\LK(\LH)$ is $F_{\sigma\delta}$.
\end{proof}

\begin{lem}[cf.~p.~48 of \cite{MR0463941}]\label{borelfns}
	If $f:\R\to\R$ is a bounded Borel function, then the map $\LB(H)_{sa}\to\LB(H)_{sa}$ given by $T\mapsto f(T)$ is Borel.
\end{lem}

\begin{proof}
	Let $(p_n)_n$ be a sequence of real polynomials converging to $f$ pointwise, which are uniformly bounded on compact sets. It follows by basic spectral theory that, for $T\in\LB(H)_{sa}$, $p_n(T)$ converges to $f(T)$ weakly. Thus, the map in questions is a pointwise (weak) limit of Borel functions, by Lemma \ref{ops_cts}, and hence Borel.
\end{proof}

\begin{lem}\label{cpct_Sch_p_Polish}
	$\LK(H)$ and $\LB^p(H)$, for $1\leq p<\infty$, are Polish spaces. In fact, they are separable Banach spaces when considered with the operator norm and $p$-norm, respectively.
\end{lem}

\begin{proof}
	It suffices to verify separability, which follows from the fact that each of the spaces considered contains $\LB_f(H)$ as a dense subset.
\end{proof}

\begin{prop}\label{B^p(H)Borel}
	For each $1\leq p<\infty$, $\LB^p(H)$ is a Polishable subspace of $\LK(H)$ in the norm topology, and a Borel subset of $\LB(H)$ in the strong operator topology.
\end{prop}

\begin{proof}
	Fix $1\leq p<\infty$. By Lemma \ref{cpct_Sch_p_Polish}, $\LB^p(H)$ is a separable Banach space under the $p$-norm. To prove that it is Polishable in $\LK(H)$, it suffices to verify that the Borel structures in both topologies coincide. This follows from the fact that $\|T\|\leq\|T\|_p$ for $T\in\LB^p(H)$, showing that the inclusion map $\LB^p(H)\to\LK(H)$ is a continuous injection.

For the second claim, the map $T\mapsto(T^*T)^{p/2}=|T|^p$ is Borel by Lemmas \ref{ops_cts} and \ref{borelfns}, and if $\{e_n:n\in\N\}$ is a fixed orthonormal basis for $H$, then $T\in\LB^p(H)$ if and only if there is an $M$, such that for all $N$, $\sum_{n=0}^N\langle|T|^p e_n,e_n\rangle<M$. Thus, $\LB^p(H)$ is Borel.
\end{proof}

\section{Equivalence relations in $\LB(H)$}

As per Lemmas \ref{Polish_std_Borel} and \ref{cpct_Sch_p_Polish}, $\LB(H)$ will be considered as a standard Borel space with the Borel structure induced by the strong operator topology, $\LB(H)_{\leq 1}$ a Polish space with the strong operator topology, and $\LK(H)$ a Polish space with the norm topology. We will consider the equivalence relations on $\LB(H)$, and their restrictions to $\LB(H)_{\leq 1}$ and $\LK(H)$, induced by the ideals $\LB_f(H)$, $\LK(H)$ and $\LB^p(H)$ for $1\leq p<\infty$, denoted (and named) as follows:
\begin{align*}
	T \equiv_f S &\quad\Leftrightarrow\quad T-S\in\LB_f(H) \quad\text{(\emph{modulo finite rank})}\\
	T \equiv_\ess S &\quad\Leftrightarrow\quad T-S\in\LK(H)  \hspace{-9pt}\quad\text{(\emph{modulo compact} or \emph{essential equivalence})}\\
	T \equiv_p S &\quad\Leftrightarrow\quad T-S\in\LB^p(H) \quad\text{(\emph{modulo $p$-class)}}, \text{ for $1\leq p<\infty$}.
\end{align*}

Fix an orthonormal basis $\{e_n:n\in\N\}$ for $H$ for the remainder of this section. Consider the map $\ell^\infty\to\LB(H)$ given by $\alpha\mapsto T_\alpha$, where $T_\alpha v = \sum_{n=0}^\infty \alpha_{n}\langle v,e_n\rangle e_n$, for $\alpha=(\alpha_n)_n\in\ell^\infty$ and $v\in H$.

\begin{lem}\label{cts_lemma}
	\begin{enumerate}
		\item The map $\alpha\mapsto T_\alpha$ is an isometric embedding $\ell^\infty\to\LB(H)$, with respect to the usual norms on these spaces, and maps $c_0$ into $\LK(H)$.
		\item The map $\alpha\mapsto T_\alpha$ is continuous $[0,1]^\N\to\LB(H)$, when $[0,1]^\N$ is endowed with the product topology, and $\LB(H)$ with the strong operator topology. Its range is contained within $\LB(H)_{\leq1}$.
	\end{enumerate}
\end{lem}

\begin{proof}
	(a) This map is the well-known isometric embedding of $\ell^\infty$ as diagonal multiplication operators on $H$ (see 4.7.6 in \cite{MR971256}). That it maps $c_0$ into $\LK(H)$ is a restatement of Proposition \ref{diag_cpct}.
	
	\noindent(b) Fix $\alpha\in[0,1]^\N$ and let $U=\{T\in\LB(H):\|(T-T_\alpha)v\|<\eps\}$, a subbasic open neighborhood of $T_\alpha$ in the strong operator topology, where $v=\sum_{n=0}^\infty a_ne_n$ and $\eps>0$. Pick $m$ such that $\sum_{n=m+1}^\infty|a_n|^2<\eps^2/2$, and let 
	\[
		V=\left\{\beta\in [0,1]^\N:\sum_{n=0}^m|\beta_n-\alpha_n|^2|a_n|^2<\eps^2/2\right\}.
	\]
	$V$ is an open neighborhood of $\alpha$ in $[0,1]^\N$. If $\beta\in V$, then
	\[
		\|(T_\beta-T_\alpha)v\|^2=\sum_{n=0}^m|\beta_n-\alpha_n|^2|a_n|^2+\sum_{n=m+1}^\infty|\beta_n-\alpha_n|^2|a_n|^2<\eps^2,
	\]
	showing that $T_\beta\in U$. It follows that the map is continuous.
\end{proof}

We can now complete the proof of Theorem \ref{thm1}.

\begin{proof}[Proof of Theorem \ref{thm1}.]
	By Propositions \ref{B_f(H)&K(H)Borel} and \ref{B^p(H)Borel}, each of the equivalence relations under consideration is Borel in the relevant spaces. We will use restrictions of the map $\alpha\mapsto T_\alpha$ to different domains, which are continuous injections in all relevant cases by Lemma \ref{cts_lemma}.
	
	\noindent(a) Let $X=\prod_{n=0}^\infty [0,\frac{1}{n+1}]$, and consider the equivalence relation $E$:
\[
	\alpha E \beta \quad\Leftrightarrow\quad \exists m\forall n\geq m(\alpha_n=\beta_n)
\]
for $\alpha,\beta\in X$. This can be identified (up to Borel brieducibility) with $E_1$. We use the restriction of the map $\alpha\mapsto T_\alpha$ to $X$. By Lemma \ref{cts_lemma}(a), it maps into $\LK(H)$. Moreover, $T_\alpha-T_\beta$ is of finite rank if and only if $\alpha E_1\beta$. Thus, $E_1\sqsubseteq_B\;\equiv_f$ on $\LK(H)$ or $\LB(H)$, and the result follows by Theorem \ref{E1}.

\noindent(b) We use the map the restriction of the map $\alpha\mapsto T_\alpha$ to $[0,1]^\N$, which maps into $\LB(H)_{\leq 1}$ by Lemma \ref{cts_lemma}(b). Suppose that $\alpha,\beta\in[0,1]^\N$, then $(T_\alpha-T_\beta)v=\sum_{n=0}^\infty(\alpha_n-\beta_n)\langle v,e_n\rangle e_n$, for $v\in H$. By Proposition \ref{diag_cpct}, $T_\alpha-T_\beta$ is compact if and only if $\alpha-\beta\in c_0$, showing $[0,1]^\N/c_0\sqsubseteq_B\;\equiv_\ess$ on $\LB(H)_{\leq 1}$ or $\LB(H)$. The result follows by Lemma \ref{transturb_restr}(a).
	
\noindent(c) We again use the restriction of $\alpha\to T_\alpha$ to $[0,1]^\N$. Fix $1\leq p<\infty$. Suppose that $\alpha,\beta\in[0,1]^\N$. For $x\in H$, we have that $|T_\beta-T_\alpha|^px=\sum_{n=0}^\infty|\beta_n-\alpha_n|^p\langle x,e_n\rangle e_n$, and so, $\sum_{n=0}^\infty \langle|T_\alpha-T_\beta|^pe_n,e_n\rangle = \sum_{n=0}^\infty|\beta_n-\alpha_n|^p$. Thus, $\alpha-\beta\in\ell^p$ if and only if $T_\alpha-T_\beta\in\LB^p(H)$, showing $[0,1]^\N/\ell^p\sqsubseteq_B\;\equiv_p$ on $\LB(H)_{\leq 1}$ or $\LB(H)$. Similarly, for the restriction to $\LK(H)$, we use the restriction of the map $\alpha\mapsto T_\alpha$ to $c_0$, and obtain $c_0/\ell^p\sqsubseteq_B\;\equiv_p$ on $\LK(H)$. The results follow as in (b), using Lemma \ref{transturb_restr}(b) in the $[0,1]^\N/\ell^p$ case, and Lemma \ref{transturb}(a) in $c_0/\ell^p$ case.
\end{proof}

\begin{proof}[Proof of Corollary \ref{cor1}.]
	If $\LB_f(H)$ was Polishable in either topology, then its translation action on $\LK(H)$ would be a Polish group action, contrary to Theorem \ref{thm1}(a).
\end{proof}

\section{Equivalence relations in $\LP(H)$}

Recall that $\LP(H)$ is the set of projections in $\LB(H)$, a Polish space in the strong operator topology by Lemma \ref{Polish_std_Borel}. Fix an orthonormal basis $\{e_n:n\in\N\}$ for $H$ throughout this section. For each $x\subseteq\N$, let $P_x$ be the projection onto the subspace $\bar{\linspan}\{e_n:n\in x\}$. Then, for $v\in H$, $P_xv =\sum_{n\in x}\langle v,e_n\rangle e_n$. The map $x\mapsto P_x$ is called the \emph{diagonal embedding} (with respect to this basis), and is the restriction to $2^\N$ of the map $\alpha\mapsto T_\alpha$ from \S 4.

\subsection{Equivalence modulo finite rank}

There are two natural ways to define equivalence modulo \emph{finite rank} or \emph{finite dimension} on $\LP(H)$. One could simply restrict $\equiv_f$ to $\LP(H)$, or one could say that $P\equiv_{fd}Q$ if there exist finite dimensional subspaces $U$ and $V$ of $H$ such that $\ran(P)\subseteq\ran(Q)+U$ and $\ran(Q)\subseteq\ran(P)+V$. In fact, these notions coincide. We will use the fact that if $V$ is a closed subspace of $H$ and $F$ a finite dimensional subspace of $H$, then $V+F$ is closed (E 2.1.4 in \cite{MR971256}).

\begin{prop}
	Let $P,Q\in\LP(H)$. The following are equivalent:
	\begin{enumerate}[\upshape (i)]
		\item $P\equiv_{fd}Q$.
		\item There exist finite dimensional subspaces $W\subseteq\ran(P)^\perp$ and $Y\subseteq\ran(Q)^\perp$ such that $\ran(P)+W=\ran(Q)+Y$.
		\item $P\equiv_f Q$.
	\end{enumerate}
\end{prop}

\begin{proof}
	(i) $\Rightarrow$ (ii): Let $U$ and $V$ witness $P\equiv_{fd}Q$ as in the definition. Let $W$ and $Y$ be the images of $U$ and $V$ under orthogonal projections onto $\ran(P)^\perp$ and to $\ran(Q)^\perp$, respectively. Then, $\ran(P)+U=\ran(P)+W$ and $\ran(Q)+V=\ran(Q)+Y$.
		
	\noindent(ii) $\Rightarrow$ (iii): For $W$ and $Y$ as in (ii), let $R$ be the projection onto $W$ and $R'$ the projection onto $Y$. Since $W$ is orthogonal to $\ran(P)$, $P+R$ is the projection onto $\ran(P)+W$. Likewise $Q+R'$ is the projection onto $\ran(Q)+Y$. Thus, $P+R=Q+R'$, and so $P-Q=R'-R$, a finite rank operator.
	
	\noindent(iii) $\Rightarrow$ (i): Suppose that $P-Q=A$ where $A\in\LB_f(A)$. Then,
	\[
		\ran(P) = \ran(Q+A)\subseteq\ran(Q)+\ran(A)=\ran(Q)+\ran(A),
	\]
	and likewise,
	\[
		\ran(Q)=\ran(P-A)\subseteq\ran(P)+\ran(A)=\ran(P)+\ran(A).
	\]
	Since $\ran(A)$ is finite-dimensional, it follows that $P\equiv_{fd}Q$.
\end{proof}

Consequently, we will use $\equiv_f$ for this (Borel, by Proposition \ref{B_f(H)&K(H)Borel}) equivalence relation on $\LP(H)$. It is easy to see that the diagonal embedding witnesses the non-smoothness of $\equiv_f$ on $\LP(H)$.

To show that $\equiv_f$ restricted to $\LP(H)$ is of higher complexity, we define a new map $[0,1]^\N\to\LP(H)$ given by $\alpha\mapsto P_\alpha$ as follows: For each $\alpha=(\alpha_n)_n\in[0,1]^\N$, let $P_\alpha$ be the projection onto $\bar{\linspan}\{e_{2n}+\alpha_ne_{2n+1}:n\in\N\}$.

\begin{lem}\label{P_alphacts}
	The map $[0,1]^\N\to\LP(H)$ given by $\alpha\mapsto P_\alpha$ is a continuous injection.
\end{lem}

\begin{proof}
	First we show that $\alpha\mapsto P_\alpha$ is injective. Let $\alpha,\beta\in[0,1]^\N$ with $\alpha\neq\beta$, so $\alpha_k\neq\beta_k$ for some $k$. In order to show that $P_\alpha\neq P_\beta$, it suffices to show that $P_\alpha(e_{2k}+\beta_ke_{2k+1})\neq e_{2k}+\beta_ke_{2k+1}=P_\beta(e_{2k}+\beta_ke_{2k+1})$. Note that
	\[
		P_\alpha(e_{2k}+\beta_ke_{2k+1})=\frac{1+\alpha_k\beta_k}{1+\alpha_k^2}(e_{2k}+\alpha_ke_{2k+1}).
	\]
	By linear independence of $e_{2k}$ and $e_{2k+1}$, the right hand side of the displayed equation is equal to the input on the left hand side if and only if $\alpha_k=\beta_k$. Thus, $P_\alpha\neq P_\beta$.

To see that the map is continuous,\footnote{We thank the anonymous referee for a much shortened proof of this fact.} for each $n\in\N$ and $\alpha\in[0,1]^\N$, let $P_{n,\alpha}$ be the projection of $H$ onto $\linspan\{e_{2n}+\alpha_n e_{2n+1}\}$. It is clear that for each $n$, the map $[0,1]^\N\to\LP(H)$ given by $\alpha\mapsto P_{n,\alpha}$ is strongly continuous, and $P_\alpha=\bigoplus_{n\in\N}P_{n,\alpha}$. To see that $\alpha\mapsto P_\alpha$ is strongly continuous, let $\alpha_k\to\alpha$ in $[0,1]^\N$, and $v$ be a unit vector. By density and the fact that $\|P_{n,\alpha}\|\leq 1$ for all $n$ and $\alpha$, it suffices to consider $v$ in the (algebraic) direct sum $\bigoplus_n\linspan\{e_{2n},e_{2n+1}\}$, in which case $\|(P_{\alpha_k}-P_{\alpha})v\|\to 0$ follows from the strong continuity of each of the factors $P_{n,\alpha}$.
\end{proof}

For $\alpha\in[0,1]^\N$, the vectors $\frac{1}{\sqrt{1+\alpha_n^2}}(e_{2n}+\alpha_ne_{2n+1})$, $n\in\N$, form an orthonormal basis for $\ran(P_\alpha)$. Thus, we can write,
\begin{align*}
	P_\alpha v 
		&=\sum_{n=0}^\infty \frac{a_{2n}+a_{2n+1}\alpha_n}{1+\alpha_n^2}(e_{2n}+\alpha_ne_{2n+1}),
\end{align*}
and
\begin{align*}
		(P_\alpha-P_\beta)v&=\sum_{n=0}^\infty\left[ \frac{a_{2n}+a_{2n+1}\alpha_n}{1+\alpha_n^2}-\frac{a_{2n}+a_{2n+1}\beta_n}{1+\beta_n^2} \right]e_{2n}\\
		&\quad + \sum_{n=0}^\infty\left[\frac{a_{2n}\alpha_n+a_{2n+1}\alpha_n^2}{1+\alpha_n^2}-\frac{a_{2n}\beta_n+a_{2n+1}\beta_n^2}{1+\beta_n^2}\right]e_{2n+1},
	\end{align*}
for $\alpha,\beta\in[0,1]^\N$ and $v=\sum_{n=0}^\infty a_ne_n\in H$. 

Since we must consider the difference $P_\alpha-P_\beta$ several times in what follows, it will be useful to have it in a canonical form. Denote by $T_0$, $T_1$, $T_2$ and $T_3$ the diagonal operators
	\begin{align*}
		T_0v&=\sum_{n=0}^\infty\left[\frac{1}{1+\alpha_n^2}-\frac{1}{1+\beta_n^2}\right]a_{2n}e_{2n},\\
		T_1v&=\sum_{n=0}^\infty\left[\frac{\alpha_n}{1+\alpha_n^2}-\frac{\beta_n}{1+\beta_n^2}\right]a_{2n+1}e_{2n+1},\\
		T_2v&=\sum_{n=0}^\infty\left[\frac{\alpha_n}{1+\alpha_n^2}-\frac{\beta_n}{1+\beta_n^2}\right]a_{2n}e_{2n},\\
		T_3v&=\sum_{n=0}^\infty\left[\frac{\alpha_n^2}{1+\alpha_n^2}-\frac{\beta_n^2}{1+\beta_n^2}\right]a_{2n+1}e_{2n+1},
	\end{align*}
	and by $S_0$ and $S_1$ the operators
	\[
		S_0v=\sum_{n=0}^\infty a_{2n+1}e_{2n} \quad\text{and}\quad S_1v=\sum_{n=0}^\infty a_{2n}e_{2n+1},
	\]
	for $v=\sum_{n=0}^\infty a_ne_n$. Each of the aforementioned operators is bounded, and by collecting terms, one can show that
	\begin{align}\label{diff_projs}
		P_\alpha-P_\beta=T_0+S_0T_1+S_1T_2+T_3.
	\end{align}
	
	We can now prove Theorem \ref{thm2}(a).

\begin{proof}[Proof of Theorem \ref{thm2}(a).]\footnote{The author is indebted to Ilijas Farah for suggesting this result and ideas of its proof.}
	By Lemma \ref{P_alphacts}, the map $\alpha\mapsto P_\alpha$ is a continuous injection. Represent $E_1$ on $[0,1]^\N$ by $\alpha E_1 \beta \Leftrightarrow \exists m\forall n\geq m(\alpha_n=\beta_n)$.
	As above, for $\alpha,\beta\in[0,1]^\N$, we have the representation $P_\alpha-P_\beta=T_0+S_0T_1+S_1T_2+T_3$. Clearly, if $\alpha E_1\beta$, then all but finitely many of the coefficients (which are independent of $v$) $\left[\frac{1}{1+\alpha_n^2}-\frac{1}{1+\beta_n^2}\right]$, $\left[\frac{\alpha_n}{1+\alpha_n^2}-\frac{\beta_n}{1+\beta_n^2}\right]$ and $\left[\frac{\alpha_n^2}{1+\alpha_n^2}-\frac{\beta_n^2}{1+\beta_n^2}\right]$ will be $0$, showing that $P_\alpha-P_\beta$ has finite rank.
	
	Conversely, suppose that $P_\alpha-P_\beta$ has finite rank. It follows that the operator $T=T_0+S_0T_1$, given by
	\[
		Tv=\sum_{n=0}^\infty\left[\frac{1}{1+\alpha_n^2}-\frac{1}{1+\beta_n^2}\right]a_{2n}e_{2n} + \sum_{n=0}^\infty\left[\frac{\alpha_n}{1+\alpha_n^2}-\frac{\beta_n}{1+\beta_n^2}\right]a_{2n+1}e_{2n}
	\]
	for $v=\sum_{n=0}^\infty a_ne_n$, is of finite rank. Using vectors of the form $\sum_{n=0}^\infty a_{2n}e_{2n}$ and $\sum_{n=0}^\infty a_{2n+1}e_{2n+1}$ it is easy to see that in order for $T$ to be finite rank, all but finitely many of the terms $\left[\frac{1}{1+\alpha_n^2}-\frac{1}{1+\beta_n^2}\right]$, and $\left[\frac{\alpha_n}{1+\alpha_n^2}-\frac{\beta_n}{1+\beta_n^2}\right]$ are $0$. Since $\alpha_n\geq 0$ and $\beta_n\geq 0$, $\frac{1}{1+\alpha_n^2}-\frac{1}{1+\beta_n^2} = 0$ if and only if $\alpha_n=\beta_n$. Thus, $\alpha E_1\beta$, and so $E_1\sqsubseteq_B\;\equiv_f$ on $\LP(H)$. The result follows by Theorem \ref{E1}.
\end{proof}

\subsection{Essential equivalence}

The last equivalence relation we wish to study is the restriction of $\equiv_\ess$ to $\LP(H)$. The quotient of $\LP(H)$ by this relation can be identified with the set of projections in Calkin algebra $\LB(H)/\LK(H)$, by Proposition 3.1 in \cite{MR2300900}.

We note that, although a projection is compact if and only if it is of finite rank, this is not true of the difference of two projections. In particular, $\equiv_\ess$ does not coincide with $\equiv_f$ on $\LP(H)$. However, as before, the diagonal embedding witnesses the non-smoothness of $\equiv_\ess$.

To prove Theorem \ref{thm2}(b), we will again use the map $\alpha\mapsto P_\alpha$ used to prove Theorem \ref{thm2}(a).

\begin{proof}[Proof of Theorem \ref{thm2}(b).]
	We claim that the map $\alpha\mapsto P_\alpha$ is a reduction of $[0,1]^\N/c_0$ to $\equiv_\ess$, from which the result will follow by Lemma \ref{transturb_restr}. Let $\alpha,\beta\in[0,1]^\N$, and suppose that $\alpha-\beta\in c_0$. We will use the representation of $P_\alpha-P_\beta$ in equation (\ref{diff_projs}). By Proposition \ref{diag_cpct}, and the inequalities
	\begin{align*}
		\left|\frac{1}{1+\alpha_n^2}-\frac{1}{1+\beta_n^2} \right|&=\left|\frac{\beta_n^2-\alpha_n^2}{(1+\alpha_n^2)(1+\beta_n^2)}\right|\leq |\beta_n-\alpha_n||\beta_n+\alpha_n|,\\
		\left|\frac{\alpha_n}{1+\alpha_n^2}-\frac{\beta_n}{1+\beta_n^2}\right| &= \left|\frac{\alpha_n+\alpha_n\beta_n^2-\beta_n-\alpha_n^2\beta_n}{(1+\alpha_n^2)(1+\beta_n^2)}\right|\\ &\leq |\alpha_n-\beta_n|+|\alpha_n||\beta_n-\alpha_n||\beta_n|,\\
		\left|\frac{\alpha_n^2}{1+\alpha_n^2}-\frac{\beta_n^2}{1+\beta_n^2}\right|&=\left|\frac{\alpha_n^2-\beta_n^2}{(1+\alpha_n^2)(1+\beta_n^2)}\right|\leq|\beta_n-\alpha_n||\beta_n+\alpha_n|,
	\end{align*}
	we have that $T_0$, $T_1$, $T_2$ and $T_3$ are compact. Since the compact operators form an ideal, $S_0T_1$ and $S_1T_2$ are also compact, and thus so is $P_\alpha-P_\beta$.
	
	Conversely, suppose that $P_\alpha-P_\beta$ is compact. We will use that if an operator is compact, then it is weak--norm continuous on the closed unit ball of $H$ (3.3.3 in \cite{MR971256}). Since the sequence $e_m$ converges \emph{weakly} to $0$ as $m\to\infty$, it follows that $(P_\alpha-P_\beta)e_{2m}$ and $(P_\alpha-P_\beta)e_{2m+1}$ converge \emph{in norm} to $0$. Observe that
	\begin{align*}
		(P_\alpha-P_\beta)e_{2m}&=\left[\frac{1}{1+\alpha_m^2}-\frac{1}{1+\beta_m^2}\right]e_{2m} + \left[\frac{\alpha_m}{1+\alpha_m^2}-\frac{\beta_m}{1+\beta_m^2}\right]e_{2m+1},\\
		(P_\alpha-P_\beta)e_{2m+1}&=\left[\frac{\alpha_m}{1+\alpha_m^2}-\frac{\beta_m}{1+\beta_m^2}\right]e_{2m} + \left[\frac{\alpha_m^2}{1+\alpha_m^2}-\frac{\beta_m^2}{1+\beta_m^2}\right]e_{2m+1},
	\end{align*}
	and so the coefficients on the right hand side of these identities converge to $0$ as $m\to\infty$.
	Using the inequalities
	\begin{align*}
		\left|\frac{1}{1+\alpha_m^2}-\frac{1}{1+\beta_m^2}\right|&=\left|\frac{\beta_m^2-\alpha_m^2}{(1+\alpha_m^2)(1+\beta_m^2)}\right|\geq \frac{1}{4}|\alpha_m-\beta_m||\alpha_m+\beta_m|,\\
		\left|\frac{\alpha_m}{1+\alpha_m^2}-\frac{\beta_m}{1+\beta_m^2}\right|&=\left|\frac{\alpha_m+\alpha_m\beta_m^2-\beta_m-\alpha_m^2\beta_m}{(1+\alpha_m^2)(1+\beta_m^2)}\right|\\&\geq\frac{1}{4}|\alpha_m-\beta_m||1-\alpha_m\beta_m|,
	\end{align*}
	the quantities on the right must also converge to $0$. For any $m$, since $\alpha_m,\beta_m\in[0,1]$, we have that
	\[
		|\alpha_m-\beta_m||\alpha_m+\beta_m|+|\alpha_m-\beta_m||1-\alpha_m\beta_m|\geq|\alpha_m-\beta_m|,
	\]
	and so $\alpha_m-\beta_m$ converges to $0$, as claimed.
\end{proof}

\section{Further questions}

We have seen in the proof of Theorem \ref{thm1} that the equivalence relations $[0,1]^\N/c_0$ and $[0,1]^\N/\ell^p$ are Borel reducible to $\equiv_\ess$ and $\equiv_p$ for $1\leq p<\infty$, respectively. We may think of $\equiv_\ess$ and $\equiv_p$ as non-commutative analogues of $\R^\N/c_0$ and $\R^\N/\ell^p$, and ask whether they are of the same complexity:

\begin{ques}
	Are the equivalence relations $\equiv_\ess$ and $\equiv_p$ on $\LB(H)$ (or $\LP(H)$) Borel reducible to $\R^\N/c_0$ and $\R^\N/\ell^p$ for $1\leq p<\infty$, respectively?
\end{ques}

The Weyl--von Neumann theorem and the work of Ando--Matsuzawa \cite{MR3427601} show that unitary equivalence modulo compact on bounded self-adjoint operators is smooth. The refinement of this given by unitary equivalence modulo Schatten $p$-class has also been studied; see \cite{MR0420323}. We ask:

\begin{ques}
	What is the Borel complexity of unitary equivalence modulo Schatten $p$-class? Is it smooth? Is it classifiable by countable structures?
\end{ques}

\subsection*{Acknowledgements}
The author would like to thank Clinton Conley, Ilijas Farah, Justin Moore, Camil Muscalu, Melanie Stam and Nik Weaver for helpful conversations and correspondences which contributed to this work, and the anonymous referee for many helpful comments which vastly improved the presentation these results. The author was partially supported by NSERC grant PGSD2-453779-2014 and NSF grant DMS-1262019.


\begin{thebibliography}{HD}



\normalsize
\baselineskip=17pt


\bibitem[AM]{MR3427601} H. Ando and Y. Matsuzawa,
\emph{The Weyl--von Neumann theorem and Borel complexity of unitary equivalence modulo compacts of self-adjoint operators},
Proc.~Roy.~Soc.~Edinburgh Sect. A {145} (2015), 1115–-1144.

\bibitem[CP]{MR0420323} R. W. Carey and J. D. Pincus,
\emph{Unitary equivalence modulo the trace class for self-adjoint operators},
Amer. J. Math. {98} (1976), 481–-514.

\bibitem[Ed]{MR0487448} G. A. Edgar,
\emph{Measurability in a Banach space},
Indiana Univ. Math. J. {26} (1977), 663--677.

\bibitem[Er]{MR0463941} J. Ernest,
\emph{Charting the operator terrain},
Mem. Amer. Math. Soc. {6}(171) (1976).

\bibitem[G]{MR2455198} S. Gao,
\emph{Invariant descriptive set theory},
Vol. 293 of Pure and Applied Mathematics, CRC Press, Boca Raton, FL, 2009.

\bibitem[HL]{Hartz_Lupini_2015} M. Hartz and M. Lupini,
\emph{The classification problem for operator algebraic varieties and their multiplier algebras},
arXiv:1508.07044.

\bibitem[H]{MR1725642} G. Hjorth,
\emph{Classification and orbit equivalence relations},
Vol. 75 of Mathematical Surveys and Monographs, Amer. Math. Soc., Providence, RI, 2000.

\bibitem[Ka]{MR2441635} V. Kanovei,
\emph{Borel equivalence relations},
Vol. 44 of University Lecture Series, Amer. Math. Soc., Providence, RI, 2008.

\bibitem[Ke]{MR1967835} A. S. Kechris,
\emph{Actions of Polish groups and classification problems},
In Analysis and Logic, Vol. 262 of London Math. Soc. Lecture Note Series, 115–-187, Cambridge Univ. Press, 2002.

\bibitem[KL]{MR1396895} A. S. Kechris and A. Louveau,
\emph{The classification of hypersmooth Borel equivalence relations},
J. Amer. Math. Soc. {10} (1997), 215--242.

\bibitem[KS]{MR1855842} A. S. Kechris and N. E. Sofronidis,
\emph{A strong generic ergodicity property of unitary and self-adjoint operators},
Ergodic Th. Dynam. Systems {21} (2001), 1459--1479.

\bibitem[O]{MR2213716} M. R. Oliver,
\emph{Borel cardinalities below $c_0$},
Proc. Amer. Math. Soc. {134} (2006), 2419--2425.

\bibitem[P]{MR971256} G. K. Pedersen,
\emph{Analysis Now},
Vol. 118 of Graduate Texts in Mathematics, Springer-Verlag, New York, 1989.

\bibitem[vN]{vN35} J. von Neumann,
\emph{Charakterisierung des spektrums eines integraloperators},
Actualit\'{e}s Sci. Indust. {229}(212) (1935), 38--55.

\bibitem[W]{MR2300900} N. Weaver,
\emph{Set theory and C*-algebras},
Bull. Symbolic Logic {13} (2007), 1--20.


%
%
%
%

\end{thebibliography}

\end{document}